\documentclass[a4paper,12pt,reqno]{amsart}
\usepackage[active]{srcltx}
\usepackage{verbatim}
\usepackage{epsfig,graphicx,color,mathrsfs}
\usepackage{graphicx}
\usepackage{amssymb}
\usepackage[english]{babel}
\usepackage[T1]{fontenc}
\usepackage[latin9]{inputenc}
\usepackage{amsmath}
\usepackage{amsthm}
\usepackage{amssymb}
\usepackage{xcolor}

\makeatletter
\theoremstyle{plain}
\newtheorem{thm}{\protect\theoremname}
\theoremstyle{plain}

\theoremstyle{plain}

\theoremstyle{plain}

\theoremstyle{plain}

\theoremstyle{plain}
\newtheorem{rem}[thm]{\protect\Remarkname}
\theoremstyle{plain}

\makeatother

\usepackage{babel}
\providecommand{\lemmaname}{Lemma}
\providecommand{\definitionname}{Definition}
\providecommand{\propositionname}{Proposition}
\providecommand{\theoremname}{Theorem}
\providecommand{\Remarkname}{Remark}
\providecommand{\conjecturename}{Conjecture}
\providecommand{\corollaryname}{Corollary}

\newcommand{\Rn}{\ensuremath  \mathbb{R}^N_+}

\thanks{$^*$ corresponding author}

\title[Classification of   solutions]{The Classification of  all weak solutions to $-\Delta u={u^{-\gamma}}$ in the half space}

\email{montoro@mat.unical.it., muglia@mat.unical.it, sciunzi@mat.unical.it}

\address{Department of Mathematics and Computer Science, UNICAL, Rende (CS), Italy}
\author[L.\ Montoro]{Luigi Montoro}
\author[L.\ Muglia]{Luigi Muglia}
\author[B.\ Sciunzi]{Berardino Sciunzi $^*$}

\begin{document}

\begin{abstract}
We provide the classification of all the positive solutions to $-\Delta u=\frac{1}{u^\gamma}$ in the half space, under minimal assumption. 
\end{abstract}
%\MSC[2020]{35J75, 35A02, 35B09}

\maketitle

\section{Introduction}
We deal with the classification of positive weak solutions to the singular problem 
\begin{equation}\tag{$\mathcal P_{\gamma}$}\label{MP}
\begin{cases}
\displaystyle -\Delta u=\frac{1}{u^\gamma} & \text{in} \,\, \mathbb{R}^N_+\\
u>0 & \text{in} \,\,  \mathbb{R}^N_+\\
u=0 & \text{on} \,\, \partial \mathbb{R}^N_+,
\end{cases}
\end{equation}
where $N\geq 1$, $\gamma >0$ and $\mathbb{R}^N_+:=\{x\in \mathbb{R}^N:x_N>0\}$. \\

\noindent
\noindent The study of semilinear elliptic problems involving singular nonlinearities has attracted many studies starting from the seminal papers \cite{CraRaTa}, see  
\cite{A,BM,Bo,Br,CDG,DP,FMSL,GUI,LM,oliva} and the references therein. 
 This kind of problems,  in bounded or unbounded domains, arise
 e.g. from the study of pseudoplastic flow in fluid mechanics, see \cite{Calle,Stuart}.
In particular, in \cite{Calle}, a solution is presented for the classical case of the incompressible flow of a uniform stream past a semi-infinite flat plate at zero incidence. Flows of this type are encountered in glacial advance and in transport of coal slurries down conveyor belts, see \cite{With}. Singular semilinear elliptic problems also appear, in a slightly different formulation $f(u)=(1-u)^{-\gamma}$, in \cite{Esp1,Esp2} motivated by the theory of
the so-called MEMS devices,
namely  devices (Micro-Electro Mechanical System)  composed by a thin dielectric
elastic membrane held fixed at level $0$ placed below an upper plate at level $1$.

\

\noindent From a mathematical point of view, the case when the domain is a half space, appears as limiting case when performing blow-up/blow-down analysis. In this case, also considering more general equations of the form $-\Delta u= u^{-\gamma}+f(u)$, the limiting problem only involves the pure negative power (in a half space) since it is the dominant part near the boundaries, see see \cite{CaEsSc, EsSc}.

\

\noindent To state our main result, let us first recall  that the solutions, in general, have not $H^1$ regularity up to the boundary (see \cite{LM}) and the equation has to be understood in the following way: $u\in H^1_{loc}(\mathbb{R}^N_+)$ and 
\begin{equation}\label{w-sol}
	\int_{\Rn}(\nabla u,\nabla \varphi)=\int_{\Rn}\frac{1}{u^\gamma}\varphi \qquad \forall \varphi\in C^1_c(\mathbb{R}_+^N).
\end{equation}
The zero Dirichlet boundary condition also has to be understood in the weak meaning
\[
(u-\varepsilon)^+\varphi\chi_{\mathbb{R}^N_+}\in H^1_0(\mathbb{R}^N_+)\qquad\forall\,\varphi\in C^1_c(\mathbb{R}^N)\,.
\]
\begin{rem}
In some cases, see e.g. \cite{Bo} for the case of bounded domains, when dealing with singular equations, it is natural to understand the zero Dirichlet datum requiring that $u^{\frac{\gamma +1}{2}}\in H^1_0(\Omega)$. We stress the fact that such assumption already implies that $(u-\varepsilon)^+\in H^1_0(\Omega)$ so that our analysis do apply also to this case. 
\end{rem}

The aim of this paper is to improve our previous results in \cite{nostro} where we provided the classification of the solutions in the case $\gamma>1$ assuming a uniform $L^\infty$ bound in the strips $\{x\in \mathbb{R}^N_+\,:\, 0<x_N<b\}$. We remove such an assumption as well as we remove any regularity assumption on the solutions. Furthermore we deduce  that, in the case $0<\gamma\leq 1$ there are no solutions. \\

\noindent Our main result is the following

\begin{thm}\label{mainthm}
	
         Let $u$ be a solution to \eqref{MP}. Then, in the case   $\gamma>1$, all the solutions $u$ are one dimensional, i.e.
	\[u(x)=u(x_N)\]  and, either
	$$
	u(x_N)=\frac{(\gamma+1)^\frac{2}{\gamma+1}}{(2\gamma-2)^\frac{1}{\gamma+1}}
	(x_N)^\frac{2}{\gamma+1}
	$$
	or
	$$
	u(x_N)=\lambda^{-\frac{2}{\gamma+1}}v(\lambda x_N) \quad \lambda>0,
	$$
	where $v(t)\in C^2(\mathbb{R}_+)\cap C(\overline{\mathbb{R}_+})$ is the unique positive solution to
	\begin{eqnarray}
		\begin{cases}
			\displaystyle - v''=v^{-\gamma} \quad & t>0\\
			v(0)=0\quad  \displaystyle &\lim_{t\to+\infty}v'(t)=1.
		\end{cases}
	\end{eqnarray}
In the case where $0<\gamma\leq 1$,  there are no solutions to \eqref{MP}. 
\end{thm}

\begin{rem}
Theorem \ref{mainthm} concludes the picture for the classification of solutions to semilinear elliptic equations in half spaces involving power nonlinearities. Actually it is well known that, in the case of positive powers ($\gamma <0$ with our notation), solutions do not exist under very general assumptions, see the recent developments in \cite{DUP1,DUP2}, and the references therein. The case of negative powers ($\gamma >0$ with our notation) still exhibit non existence of the solutions for the case $0<\gamma\leq 1$. On the contrary, if $\gamma >1$ solutions exist and are classified by our Theorem \ref{mainthm}.
\end{rem}

\noindent Our proof is based on the following steps:\\

\

\noindent \textbf{\textit{Step 1}}: we start proving the local boundedness of the solutions in $\overline{\mathbb{R}_+}$. We follow the clever technique of J. Serrin \cite{serrinacta} to carry out the Moser iteration scheme. Some care is needed since solutions have not $H^1$-regularity up to the boundary.\\

\noindent \textbf{\textit{Step 2}}:  we exploit the  barriers technique of \cite{DP,GUI} to deduce that, near the boundary, the solution is trapped among to continuous functions and, therefore, is continuous up to the boundary. \\

\noindent \textbf{\textit{Step 3}}: we prove that the solution has 1-D symmetry. Here we develop a completely different proof that the one in \cite{nostro}, based on the use of the Kelvin transformation. Again the lack of regularity of the solutions up to the boundary is the crucial issue. \\

\noindent \textbf{\textit{Conclusion}}: the proof of the main result is consequently reduce to a ODE analysis.

\section{Regularity of the solutions}
We start this section studying the local boundedness of the solutions. Let us consider $H^1_{loc}(\mathbb{R}^N_+)$ weak solutions to the problem that fulfil the Dirichlet datum in the generalized meaning so that
\[
(u-\varepsilon)^+\in H^1_{loc}(\overline{\mathbb{R}^N_+})\,.
\]
The proof that follows is based on a standard use of the Moser iteration scheme. 
There is here an added difficulty given by the singular nature of the problem that we overcome 
 following closely the technique developed by J. Serrin in \cite{serrinacta}. 

\begin{thm}\label{localbound}
If  $u\in H^1_{loc}(\mathbb{R}^N_+)$ is a weak solution problem $\mathcal P_{\gamma}$, then $u$ is locally bounded.
\end{thm}

\begin{proof}
	For $\varepsilon >0$ fixed, set
	\[
	u_\varepsilon\,:=\, (u-\varepsilon)^+\,+\,\frac{1}{\varepsilon^\gamma}\,.
	\]
	Furthermore, for $q\geq 1$ and $l>\frac{1}{\varepsilon^\gamma}$, we define 
	\begin{equation}
	F(u_\varepsilon)\,:=\, \begin{cases}
		u_\varepsilon^q\qquad\qquad \qquad\qquad\qquad if\quad \frac{1}{\varepsilon^\gamma}\leq u_\varepsilon\leq l\\
		ql^{q-1}u_\varepsilon-(q-1)l^q\qquad\quad  if \quad u_\varepsilon\geq l
	\end{cases}
	\end{equation}
and, consequently, 
\begin{equation}
	G(u)\,:=\, F(u_\varepsilon)F'(u_\varepsilon)-\frac{q}{\varepsilon^{\beta\gamma}}
\end{equation}
where $\beta$ is such that $2q=\beta+1$. We now indicate by $\eta$ a standard smooth cut-off function and consider the test function
\[
\varphi\,:=\,G(u)\eta^2\,.
\]
It is easy to check that this is a suitable test function for our problem. Note in fact that $(u-\varepsilon)^+\in H^1_{loc}(\overline{\mathbb{R}^N_+}) $ by the definition of the Dirichlet boundary datum and $G(u)$ is linear for $u_\varepsilon>l$. Consequently we deduce that
\begin{equation}\nonumber
\int_{\mathbb{R}^N_+}\, |\nabla u|^2G'(u)\eta^2=\int_{\mathbb{R}^N_+} -2G(u)\eta(\nabla u\,,\nabla\eta)+\frac{G(u)\eta^2}{u^\gamma}
\end{equation}
so that 
\begin{equation}\nonumber
	\begin{split}
		\int_{\mathbb{R}^N_+}\, |\nabla u|^2G'(u)\eta^2&\leq 2 \int_{\mathbb{R}^N_+} G(u)\eta|\nabla u||\nabla\eta|+\frac{G(u)\eta^2}{u^\gamma} \\
		& \leq 2 \int_{\mathbb{R}^N_+} G(u)\eta|\nabla u||\nabla\eta|+ \int_{\mathbb{R}^N_+}\frac{G(u)\eta^2}{\varepsilon ^\gamma} \\
		&\leq 2 \int_{\mathbb{R}^N_+} G(u)\eta|\nabla u||\nabla\eta|+ \int_{\mathbb{R}^N_+}G(u)\eta^2u_\varepsilon\,. \\
	\end{split}
\end{equation}
Taking into account the definition of $F(\cdot)$ and $G(\cdot)$, it is easily seen that
	\begin{equation}\nonumber
	G'(u)\,:=\, \begin{cases}\displaystyle 
		\frac{\beta}{q}\,(F')^2
\qquad\,\, \,\,\, if\quad (u-\varepsilon)<l-\frac{1}{\varepsilon^\gamma} \\\\
		\,\,\,\,\displaystyle (F')^2\qquad\quad  if \quad (u-\varepsilon)>l-\frac{1}{\varepsilon^\gamma}
	\end{cases}
\end{equation}
and $G\leq FF'$. Thus, setting
\[
v(x)\,:=\,F(u_\varepsilon)
\]
and observing that $u_\varepsilon\,F'\leq q F$,  we get that
\begin{equation}\nonumber
	\begin{split}
	\min 	\left\{1,\frac{\beta}{q}\right\}\int_{\mathbb{R}^N_+}|\nabla v|^2\eta^2
		&= \min 	\left\{1,\frac{\beta}{q}\right\}\int_{\mathbb{R}^N_+}|\nabla u|^2(F'(u_\varepsilon))^2\eta^2\\
		&\leq \int_{\mathbb{R}^N_+}|\nabla u|^2G'(u)\eta^2\\
		& \leq 2 \int_{\mathbb{R}^N_+} FF'\eta|\nabla u||\nabla\eta|+ q\int_{\mathbb{R}^N_+}v^2\eta^2\,. \\
	\end{split}
\end{equation}
Recalling that $\nabla v=F'(u_\varepsilon)\nabla u$, 
by a simple use of the Young inequality, we deduce that 
\begin{equation}\label{stimagrad}
\int_{\mathbb{R}^N_+}|\nabla v|^2\eta^2
	\leq C \int_{\mathbb{R}^N_+} |\nabla\eta|^2v^2+ C \,q\int_{\mathbb{R}^N_+}v^2\eta^2\,. 
\end{equation}
The Sobolev inequality now provides that
\begin{equation}\nonumber
	\left(\int_{\mathbb{R}^N_+} (v\eta)^{2^*}\right)^{\frac{1}{2^*}}\leq 
	\left(\int_{\mathbb{R}^N_+}|\nabla v|^2\eta^2
	+ |\nabla\eta|^2v^2+2|\nabla\eta||\nabla v|\eta v\right)^{\frac{1}{2}}
\end{equation}
so that, using again the Young inequality and \eqref{stimagrad}, we get
\begin{equation}\nonumber
	\left(\int_{\mathbb{R}^N_+} (v\eta)^{2^*}\right)^{\frac{1}{2^*}}\leq 
	Cq\, \left(\int_{\mathbb{R}^N_+}v^2|\nabla \eta|^2\right)^{\frac{1}{2}}+
	Cq\left(\int_{\mathbb{R}^N_+}v^2\eta^2\right)^{\frac{1}{2}}\,.
\end{equation}
This estimate recovers the ones in the proof of Theorem 1 of \cite{serrinacta} (see equation (19)). Consequently the remaining part of the proof can be carried over repeating verbatim the argument of Serrin, thus proving the local boundedness of $u_\varepsilon$ that implies the local boundedness of $u$. 
\end{proof}
Standard elliptic estimates (see \cite{GT})  show that the solution is smooth in the interior of the half space. Furthermore, the continuity of the solution in the interior of the domain follows by standard regularity estimates, since the right hand side is bounded there (see e.g. \cite{GT}).
From Lemma 4 in \cite{nostro}, and for $\gamma>1$, we know that
there exists a constant $C=C(\gamma)$ such that 
\[u \geq C x_N^{\frac{2}{\gamma+1}}\quad \text{in}\,\, \mathbb R^N_+.\]
The case $0<\gamma\leqslant 1$ is easier to study using the first eigenfunction as subsolution as in \cite{nostro} (following \cite{LM}).

We can therefore exploit the barriers arguments in  \cite{DP,GUI} to deduce the following:

\begin{thm}[\cite{DP,GUI}]\label{continuity}
	If  $u\in H^1_{loc}(\mathbb{R}^N_+)$ is a weak solution problem $\mathcal P_{\gamma}$, then $u\in C^0(\overline{\mathbb{R}^N_+})$.
\end{thm}
Actually the results of \cite{DP,GUI} are not  stated in the case of the half space. In any case the arguments do apply since they are based on local tools. The local boundedness is actually implicitly used in the technique.

\section{ 1-D symmetry}\label{sezpde}

We start this section recalling the Kelvin transformation of the solutions $u$ given by:
\[
\hat u(x)\,:=\,\frac{1}{|x|^{N-2}}u\left(\frac{x}{|x|^2}\right)\qquad x\neq0\,.
\]
It is known that 
\[
\Delta \hat u=\frac{1}{|x|^{N+2}}\Delta u | _ \frac{x}{|x|^2},  \qquad x\neq0\,.
\]
so that
\begin{equation}\nonumber
-\Delta \hat u\,=\,\frac{1}{|x|^{\beta_\gamma}}\frac{1}{\hat u^\gamma},\qquad x\neq0\qquad\text{for}\,\,\,\beta_\gamma=N+2+\gamma(N-2)\,.
\end{equation}
Such equation is classical in the interior of the half space since $u$ is smooth there. On the contrary the transformation produces a possible singularity at zero and, in any case, $\hat u$ has not $H^1$ regularity up to the boundary as well as the starting solution $u$. Consequently the equation is still understood as follows:

\begin{equation}\label{kelvin}
	\int_{\Rn}\nabla \hat u \cdot\nabla \varphi=\int_{\Rn}\frac{1}{|x|^{\beta_\gamma}}\frac{1}{\hat u^\gamma}\varphi, \qquad \forall \varphi\in C^1_c(\mathbb{R}_+^N).
\end{equation}
We are ready to prove
\begin{thm}\label{1D}
	If  $u\in H^1_{loc}(\mathbb{R}^N_+)$ is a weak solution problem $\mathcal P_{\gamma}$, then $u$ has 1-D symmetry, namely $$u=u(x_N).$$
\end{thm}
\begin{proof}
By Theorem \ref{localbound} and  Theorem \ref{continuity} we know that $u$ is locally bounded and continuous up to the boundary. Furthermore $u$ is smooth in the interior of the half space by standard elliptic estimates. 
For any $\lambda<0$, we set
\[
\Sigma_\lambda\,:=\, \{x\in \mathbb R^N_+\,:\,  x_1< \lambda\}\,
\]
adopting here 
 the notation  $x=(x_1,x')$ with $x'\in \mathbb R^{N-1}$. Furthermore we set
\[
x_\lambda=(2\lambda-x_1,x')
\]
\[
\hat u_\lambda(x)=\hat u(x_\lambda)=
\hat u(2\lambda-x_1,x') \quad\text{in}\quad 
\Sigma_\lambda\,.
\]
By the reflection invariance of the Laplace operator, we deduce that
\begin{equation}\label{kelvinlambda}
	\int_{\Rn}\nabla \hat u_\lambda \cdot\nabla \varphi=\int_{\Rn}\frac{1}{|x_\lambda|^{\beta_\gamma}}\frac{1}{\hat u_\lambda^\gamma}\varphi \qquad \forall \varphi\in C^1_c(\mathbb{R}_+^N).
\end{equation}
We stress that the transformation produces a possible singularity for $u$ at the origin and, consequently, a possible singularity for $\hat u$ at $0_\lambda$  the reflected point of the origin, namely $0_\lambda=(-2\lambda,x')$.
Now we consider the cut-off functions $0\leq \phi_R,\varphi_\varepsilon\leq 1$ such that
\begin{equation}\label{troncar}
	\begin{cases}
		\phi_R=1 & \text{in}\,\, B_R(0)\\
		\phi_R=0 & \text{in}\,\, \mathbb R^{N}\setminus B_{2R}(0)\\
		\displaystyle |\nabla\phi_R |\leq \frac{C}{R},& \text{in}\,\, B_{2R}(0)\setminus B_{R}(0)\,.
	\end{cases}
\end{equation}
In dimension $N=2$ we will exploit the fact that, by a capacitary argument, namely considering a suitable truncation of the fundamental solution,   it is possible to assume that
$$
\int_{B_{R^2}(0)\setminus B_{R}(0) }|\nabla \phi_R|^2\underset{R\rightarrow \infty}{\longrightarrow}\,0\,.
$$
For the cut-off function $\varphi_\varepsilon$ we assume that:
\begin{equation}\label{tronca0}
	\varphi_{\varepsilon}=\begin{cases}
		0 & \text{in}\,\, B_{\varepsilon^2}(0_\lambda)\\
		1 & \text{in}\,\, \mathbb R^{N}\setminus B_{\varepsilon}(0_\lambda).
	\end{cases}
\end{equation}
Again by a capacitary argument (and here in any dimension) we may and do assume that
\[
\int_{B_{\varepsilon}(0_\lambda)\setminus B_{\varepsilon^2}(0_\lambda)
}|\nabla\varphi_\varepsilon|^2\,\underset{\varepsilon\rightarrow 0}{\longrightarrow}\,0\,.
\]
Given $\tau>0$ we now consider
\[
w=w_{\varepsilon,\tau,R,\lambda}\,:=\, (\hat u-\hat u_\lambda-\tau)^+\varphi_\varepsilon^2\phi_R^2\,.
\]
Since, by Theorem \ref{continuity} $u$ is continuous up to the boundary, then $\hat u$ is continuous up to the boundary too, for $x\neq 0$. Consequently, recalling the construction of the cut-off functions,  $w$ is a $H^1$-function with its support that is compactly contained in the half space.  Consequently, by standard density arguments,   we can use it  as test function in \eqref{kelvin} and in \eqref{kelvinlambda}. Thus we get that
\begin{eqnarray*}
	&&\int_{\Sigma_\lambda}(\nabla \hat u,\nabla w)=\int_{\Sigma_\lambda}\frac{1}{|x|^{\beta_\gamma}}\frac{1}{\hat u^\gamma} w,\\
	&&\int_{\Sigma_\lambda}(\nabla \hat u_\lambda,\nabla w)=\int_{\Sigma_\lambda}\frac{1}{|x_\lambda|^{\beta_\gamma}}\frac{1}{\hat u_\lambda^\gamma} w,
\end{eqnarray*}
so that 
\begin{eqnarray*}
	\int_{\Sigma_\lambda}(\nabla (\hat u-\hat u_\lambda),\nabla w)=\int_{\Sigma_\lambda}\left(\frac{1}{|x|^{\beta_\gamma}}\frac{1}{\hat u^\gamma} 
	-\frac{1}{|x_\lambda|^{\beta_\gamma}}\frac{1}{\hat u_\lambda^\gamma} \right)w\leq 0\,,
\end{eqnarray*}
since $\hat u\geq\hat u_\lambda$ on the support of $w$ and $|x_\lambda|\leq|x|$ in $\Sigma_\lambda$.\\

\noindent We now set
\[
z_\tau\,:=\, (\hat u-\hat u_\lambda-\tau)^+
\] 
and, by direct computation, we deduce that
\begin{equation}\nonumber
	\int_{\Sigma_\lambda}|\nabla z_\tau|^2\varphi_\varepsilon^2\phi_R^2\leq 
		2\int_{\Sigma_\lambda}z_\tau|\nabla z_\tau|\varphi_\varepsilon^2\phi_R|\nabla \phi_R|+
		2\int_{\Sigma_\lambda}z_\tau|\nabla z_\tau|\varphi_\varepsilon\phi_R^2|\nabla \varphi_\varepsilon|\,.
\end{equation}
A simple application of the Young inequality
\[
8ab\leq a^2+16b^2
\] allows us to deduce that
\begin{equation}\nonumber
	\begin{split}
	\int_{\Sigma_\lambda}|\nabla z_\tau|^2\varphi_\varepsilon^2\phi_R^2\leq 
	& 8\int_{\Sigma_\lambda\cap (B_{2R}(0)\setminus B_{R}(0))}z_\tau^2\varphi_\varepsilon^2|\nabla \phi_R|^2+\\
	& 8\int_{\Sigma_\lambda\cap B_{2\varepsilon}(0_\lambda)\setminus B_{\varepsilon}(0_\lambda)} z_\tau^2\phi_R^2|\nabla \varphi_\varepsilon|^2\,\\
	&=I_1(R)+I_2(\varepsilon)\,.
\end{split}
\end{equation}
Note now that
\[
(z_\tau)^2\leq \hat u^2\leq \frac{C}{|x|^{2(N-2)}}\qquad\text{and}\qquad |\nabla\phi_R|\leq\frac{C}{R}
\]
so that, observing that $\mathcal{L}\,(B_{2R}(0)\setminus B_{R}(0))=O(R^N)$, it follows
\[
\underset{R\rightarrow +\infty}{\lim }I_1(R)=0,
\]
if $N>2$. 
If else $N=2$ we just note that $z_\tau$ is bounded far from the origin. Furthermore, as recalled here above, by a capacitary argument,  $\phi_R$ can be chosen in such a way that $\displaystyle \int_{B_{R^2}(0)\setminus B_{R}(0) }|\nabla \phi_R|^2=o(1)$ so that again we recover that $I_1(R)$
goes to zero as $R$ tends to $\infty$. On the other hand 
\[
\int_{\Sigma_\lambda\cap B_{\varepsilon}(0_\lambda)\setminus B_{\varepsilon^2}(0_\lambda)} z_\tau^2\phi_R^2|\nabla \varphi_\varepsilon|^2\leq C \int_{\Sigma_\lambda\cap B_{\varepsilon}(0_\lambda)\setminus B_{\varepsilon^2}(0_\lambda)} |\nabla \varphi_\varepsilon|^2=o(1)
\]
since $z_\tau$ is bounded in $B_{\varepsilon}(0_\lambda)\setminus B_{\varepsilon^2}(0_\lambda)$. Note that here we exploit the fact that $z_\tau\leq \hat u$ by construction. Consequently we also proved that 
\[
\underset{\varepsilon\rightarrow 0}{\lim }\,\,I_2(\varepsilon)=0\,.
\]
Concluding, passing to the limit, we get that
\[
\int_{\Sigma_\lambda}|\nabla z_\tau|^2=0
\]
causing that, by the arbitrariness  of $\tau>0$, it follows that
\[
\hat u\leq \hat u_\lambda\qquad \text{in}\quad \Sigma_\lambda\quad\text{for any}\,\,\lambda<0\,.
\]
Repeating the argument in the opposite $(-x_1)$-direction we conclude that, necessarily,
\[
\hat u (x_1,x')=\hat u (-x_1,x')\,.
\]
By the definition of the Kelvin transformation, this information
 translates into the same symmetry property for $u$, namely:
\[
 u (x_1,x')= u (-x_1,x')\,.
\]
Taking into account the translation invariance of the equation $-\Delta u=u^{-\gamma}$ with respect to the $x_1$-direction, it is  easy to see that $u$ has to be symmetric with respect to any hyperplane $\{x_1=\lambda \}$ which causes that necessarily $u$ is constant in the $x_1$-direction. Exploiting now the rotation invariance of the problem with respect to any direction parallel to the boundary, we finally deduce that:
\[
u = u (x_N)\,.
\]
\end{proof}

\begin{proof}[Proof of Theorem \ref{mainthm}] Once we know that the solutions have 1-D symmetry the proof of the result, in the case $\gamma >1$, follows by the ODE analysis carried out in \cite{nostro}. \\
	
\noindent Let us skip therefore to the case $0<\gamma\leq 1$ to prove that, in this case, there are no global solutions. We analyse the equation
\[
-u''(t)=\frac{1}{u(t)^\gamma}\qquad \text{for}\quad t>0
\] 
with $u(0)=0$. We write the identity $-u''(t)u'(t)=\frac{u'(t)}{u(t)^\gamma}$ and integrate from $1$ to $t>1$ to deduce that
\begin{equation}\label{trucc}
\frac{	1}{2} (u'(t))^2= \frac{	1}{2} (u'(1))^2+F(u(1))-F(u(t))
\end{equation}
where $F(\cdot)$ is a diverging function since $\gamma\leq 1$. By convexity we know that $u'(t)$ is a decreasing function. Therefore necessarily $u'(t)>0$ for any $t>0$ because otherwise the solution cannot be global and the claim is proved. Consequently $u(t)$ is monotone increasing.
If now 
\[
\underset{t\rightarrow +\infty}{\lim}\, u(t) =+\infty
\]
we easily get a contradiction by \eqref{trucc} and the fact that $\underset{u\rightarrow +\infty}{\lim}\, F(u) =+\infty$. If else 
\[
\underset{t\rightarrow +\infty}{\lim}\, u =L>0,
\]
we observe that there would exists a sequence of points $t_n$ such that $\underset{n\rightarrow +\infty}{\lim}\, u''(t_n)=0$, which is not possible since the right hand side of the equation do not approaches zero at such a sequence of points. This proves the claim and the theorem.
\end{proof}

\vspace{1cm}

%\begin{center}{\bf Acknowledgements}\end{center}  
%L. Montoro and B. Sciunzi are partially supported by PRIN project 2017JPCAPN (Italy): Qualitative and quantitative aspects of nonlinear PDEs, and L. Montoro by  Agencia Estatal de Investigaci\'on (Spain), project PDI2019-110712GB-100.

\

\begin{center}
{\sc Data availability statement}

\

All data generated or analyzed during this study are included in this published article.
\end{center}

\

\begin{center}
	{\sc Conflict of interest statement}
	
	\
	
	The authors declare that they have no competing interest.
\end{center}

\bibliographystyle{elsarticle-harv}

\end{document}